\documentclass{article}

\usepackage{amsmath,amssymb,amsthm,mathrsfs}
\usepackage{color}

\newtheorem{theorem}{Theorem}
\newtheorem{proposition}{Proposition}
\newtheorem{lemma}{Lemma}

\theoremstyle{definition}\newtheorem{definition}{Definition}
\theoremstyle{definition}\newtheorem{remark}{Remark}
\theoremstyle{definition}\newtheorem{example}{Example}

\def\di{\displaystyle}

\def\N{\mathbb{N}}
\def\R{\mathbb{R}}

\def\T{\mathbb{T}}
\def\I{\mathbb{I}}
\def\AC{\mathrm{AC}}
\def\L{\mathrm{L}}

\def\DD{\Delta}

\def\CC{\mathscr{C}}

\def\RR{\mathscr{R}}

\def\GG{\mathscr{G}}
\def\KK{\mathcal{K}}
\def\ab{a \trianglelefteq t_0 \trianglelefteq b}

\DeclareMathOperator*{\supess}{sup\,ess}

\newcommand{\fonction}[5]{\begin{array}[t]{lrcl}#1 :&#2 &\rightarrow &#3\\&#4& \mapsto &#5 \end{array}}

\newcommand{\fonctionsansdef}[3]{#1 : #2 \rightarrow #3}

\title{General Cauchy-Lipschitz theory for shifted and non shifted $\DD$-Cauchy problems on time scales}

\author{Lo\"ic Bourdin\footnote{Laboratoire de Math\'ematiques et de leurs Applications - Pau (LMAP). UMR CNRS 5142. Universit\'e de Pau et des Pays de l'Adour. \texttt{bourdin.l@univ-pau.fr}} ,
Emmanuel Tr\'elat\footnote{Universit\'e Pierre et Marie Curie (Univ. Paris 6) and Institut Universitaire de France, CNRS UMR 7598, Laboratoire Jacques-Louis Lions, F-75005, Paris, France. \texttt{emmanuel.trelat@upmc.fr}}
}

\date{}

\begin{document}

\maketitle

\begin{abstract}
This article is devoted to completing some aspects of the classical Cauchy-Lipschitz (or Picard-Lindel\"of) theory for general nonlinear systems posed on time scales, that are closed subsets of the set of real numbers. Partial results do exist but do not cover the framework of general dynamics on time-scales encountered e.g. in applications to control theory.
In the present work, we first introduce the notion of absolutely continuous solution for shifted and non shifted $\DD$-Cauchy problems, and then the notion of a maximal solution. We state and prove a Cauchy-Lipschitz theorem, providing existence and uniqueness of the maximal solution of a given $\DD$-Cauchy problem under suitable assumptions like regressivity and local Lipschitz continuity, and discuss some related issues like the behavior of maximal solutions at terminal points.
\end{abstract}

\noindent\textbf{Keywords:} Time scale; Cauchy-Lipschitz (Picard-Lindel\"of) theory; existence; uniqueness; shifted problems.

\noindent\textbf{AMS Classification:} 34N99; 34G20; 39A13; 39A12.

\tableofcontents

\section{Introduction}\label{section0}
The \textit{time scale} theory was introduced by S. Hilger in his PhD thesis \cite{hilg} in 1988 in order to unify discrete and continuous analysis, with the general idea of extending classical theories on an arbitrary non empty \textit{closed} subset $\T$ of $\R$. Such a closed subset $\T$ is called a \textit{time scale}. The objective is to establish the validity of some results both in the continuous case $\T =\R$ and in the purely discrete case $\T =\N$, but also to treat more general models of processes involving both continuous and discrete time elements. We refer the reader e.g. to \cite{gama,may} where the authors study a seasonally breeding population whose generations do not overlap or to \cite{ati} for applications to economy. By considering $\T = \{ 0 \} \cup \lambda^\N$ with $0 < \lambda < 1$, time scale concept also allows to cover quantum calculus \cite{kac}. Since S. Hilger defined the $\DD$-derivative and the $\DD$-integral on a time scale, many authors have 
 extended to time scales various results from the continuous or discrete standard calculus theory. We refer the reader to the surveys \cite{agar2,agar3,bohn,bohn3}. However, due to the recency of the field, the basic nonlinear theory is yet to be developed and refined.

Some Cauchy-Lipschitz (Picard-Lindel\"of) type results on time scales are provided in \cite{bohn,cich2,hilg2,kaym,kubi,laks} where the authors prove the existence and uniqueness of solutions for $\DD$-Cauchy problems of the form:
\begin{equation}\label{introcp}
q^\DD = f(q,t), \quad q(t_0) = q_0,
\end{equation}
where $t_0 \in \T$. Note that papers are devoted to $\DD$-Cauchy problems with parameter in \cite{hils4} and with time delays in \cite{karp}. Many authors are also interested in shifted $\DD$-Cauchy problems
\begin{equation}\label{introcps}
q^\DD = f(q^\sigma,t),\quad q(t_0) = q_0,
\end{equation}
where $q^\sigma = q \circ \sigma$ (see further for the precise definitions of these notions). Such shifted problems are often used as models in the existing literature (see e.g. \cite{torr8,hils2,torr7}, \cite[Remark 3.9]{hils3} and \cite[Remark 3.6]{hils4}), because they emerge in adjoint equations accordingly to the \textit{shifted} Leibniz formula \cite{bohn}
\begin{equation}
(q_1 q_2)^\DD = q_1^\DD q_2^\sigma + q_1 q_2^\DD = q_1^\DD q_2 + q_1^\sigma q_2^\DD.
\end{equation}
Nevertheless, to the best of our knowledge, there does not exist a general Cauchy-Lipschitz theory on time scales that is fully complete in order to be applied to problems arising for example in control theory\footnote{Actually, the present article was motivated by the needs of completing the existing results on Cauchy-Lipschitz theory on time scales, in order to investigate general non linear control systems on time scales, and more precisely to derive a general version of the Pontryagin Maximum Principle in optimal control.}.

Let us recall briefly the bibliographical context on the Cauchy-Lipschitz theory on time scales. The first result on $\DD$-Cauchy problems is due to S. Hilger in \cite[Paragraph 5]{hilg2}, who derived the existence and uniqueness of $\CC^1_\mathrm{rd}$-solutions for continuous dynamics. This framework is not suitable for general control problems where controls are measurable functions that have discontinuities in general. Note that similar frameworks and results are provided in \cite[Paragraph 8.2]{bohn}, in \cite{kaym,laks,tisd} and references therein. In \cite{cich2,kubi}, the authors respectively treat weak continuous and Carath\'eodory dynamics living in a general Banach space. Note that they only treat the non shifted case where $q_0$ is an initial condition, that is, solutions are only defined for $t \geq t_0$. In view of considering adjoint equations, it is of interest to study backward $\DD$-Cauchy problems where $q_0$ is a final condition, for which solutions are def
 ined for $t \leq t_0$. As is very well known in time scale calculus, the solvability of such backward non shifted $\DD$-Cauchy problems requires a \textit{regressivity} assumption on the dynamics (see e.g. \cite{bohn,hilg2} and \cite[Remark 3.8]{hils3}). This important issue is not addressed in these two articles. Another issue which is not addressed is the fact that the usual Cauchy-Lipschitz theory treats Cauchy problems constraining the solutions to take values in an open subset $\Omega$ of $\R^n$ (see e.g. \cite{codd,smal}). Finally, up to our knowledge, the notion of extension of a solution on time scales, and the behavior of the maximal solution at terminal points, have not been studied. Similarly, we are not aware of articles treating both shifted and non shifted general nonlinear $\DD$-Cauchy problems.

This article is thus devoted to fill an existing gap of the literature, and to provide a general Cauchy-Lipschitz theory on time scales generalizing the basic notions and results of the classical continuous theory surveyed e.g. in \cite{codd,smal}. Precisely, we first introduce the notion of an \textit{absolutely continuous solution}. Then we define the concept of \textit{extension} of a solution, and of \textit{maximal} and \textit{global} solutions in the time scale context. We establish a general version of the Cauchy-Lipschitz theorem (existence and uniqueness of the maximal solution, also referred to as Picard-Lindel\"of theorem) for dynamics posed on a time scale, under regressivity and local Lipschitz continuity assumptions, for shifted and non shifted general nonlinear $\DD$-Cauchy problems in the following framework:
\begin{itemize}
\item $f$ is a general $\DD$-Carath\'eodory function, where $\DD$-measure $\mu_\DD$ on a time scale $\T$ is defined in terms of Carath\'eodory extension in \cite[Chapter 5]{bohn3};
\item $q_0$ is not necessarily an initial or a final condition;
\item the solutions take their values in an open subset $\Omega$ of $\R^n$.
\end{itemize}
We also investigate the globality feature of the maximal solution. Our results are established first for general non shifted $\DD$-Cauchy problems \eqref{introcp} and then for shifted ones \eqref{introcps}.

Our study uses the work of A. Cabada and D. Vivero in \cite{caba2}, who proved a criterion for absolutely continuous functions written as the $\DD$-integral of their $\DD$-derivatives. Their result allows us to give a $\DD$-integral characterization of the solutions of $\DD$-Cauchy problems which is instrumental in our proofs.

Notice that analogous results on $\nabla$-Cauchy problems ($\rho$-shifted or not) can be  derived in a similar way.

The article is structured as follows. Section~\ref{section1} is devoted to recall basic notions of time scale calculus. In Section~\ref{section2}, we define the notions of a solution, of an extension of a solution, of a maximal and a global solution for general non shifted $\DD$-Cauchy problems. Under suitable assumptions on the dynamics, we establish a Cauchy-Lipschitz theorem and then investigate the behavior of the maximal solution at its terminal points. Section~\ref{section3} is devoted to establish similar results for \textit{shifted} $\DD$-Cauchy problems.

\section{Preliminaries on time scale calculus}\label{section1}
In this section, we recall basic results in time scale calculus. The first part concerns the structure of time scales and the notion of $\DD$-differentiability (see \cite{bohn}). The second part concerns the $\DD$-Lebesgue measure defined in terms of Carath\'eodory extension (see \cite{bohn3,guse}) and surveys results on $\DD$-integrability proved in \cite{caba}. The last part gathers properties of absolutely continuous functions borrowed from \cite{caba2}.

Let $n\in\N^*$. Throughout, the notation $\Vert \cdot \Vert$ stands for the Euclidean norm of $\R^n$. For every $x \in \R^n$ and every $R \geq 0$, the notation $\overline{B}(x,R)$ stands for the closed ball of $\R^n$ centered at $x$ and with radius $R$.

\subsection{Time scale and $\DD$-differentiability}\label{section11}
Let $\T$ be a time scale, that is, a closed subset of $\R$. We assume that $\mathrm{card}(\T) \geq 2$. For every $A \subset \R$, we denote $A_\T = A \cap \T$. An interval of $\T$ is defined by $I_\T$ where $I$ is an interval of $\R$. 

The backward and forward jump operators $\rho,\sigma:\T\rightarrow\T$ are respectively defined by
\begin{equation*}
\begin{split}
\rho (t) &= \sup \{ s \in \T\ \vert\ s < t \},\\
\sigma (t) &= \inf \{ s \in \T\ \vert\ s > t \},
\end{split}
\end{equation*}
for every $t \in \T$,
where $\rho (\min \T) = \min \T$ (resp. $\sigma(\max \T) = \max \T$) whenever $\T$ admits a minimum (resp. a maximum). 

A point $t \in \T$ is said to be a left-dense (respectively, left-scattered, right-dense or right-scattered) point of $\T$ if $\rho (t) = t$ (respectively, $\rho (t) < t$, $\sigma (t) = t$ or $\sigma (t) > t$). The graininess function $\mu:\T\rightarrow\R^+$ is defined by $\mu(t) = \sigma (t) -t$ for every $t \in \T$. 

We set $\T^\kappa = \T \backslash \{ \max \T \}$ whenever $\T$ admits a left-scattered maximum, and $\T^\kappa = \T$ otherwise. A function $q:\T\rightarrow\R^n$ is said to be $\DD$-differentiable at $t \in \T^\kappa$ if the limit
$$q^\DD (t) = \lim\limits_{\substack{s \to t \\ s \in \T}} \dfrac{q^\sigma (t) -q(s)}{\sigma (t) -s} $$
exists in $\R^n$, where $q^\sigma = q \circ \sigma$. We recall the following well known results (see \cite{bohn}):
\begin{itemize}
\item if $t \in \T^\kappa$ is a right-dense point of $\T$, then $q$ is $\DD$-differentiable at $t$ if and only if the limit
$$q^\DD (t) = \lim\limits_{\substack{s \to t \\ s \in \T}} \dfrac{q(t)-q(s)}{t-s} $$
exists in $\R^n$;
\item if $t \in \T^\kappa$ is a right-scattered point of $\T$ and if $q$ is continuous at $t$, then $q$ is $\DD$-differentiable at $t$, and
$$q^\DD (t) = \dfrac{q^\sigma(t) - q(t)}{\mu(t)}.$$
\end{itemize}

\subsection{Lebesgue $\DD$-measure and Lebesgue $\DD$-integrability}\label{section12}
Recall that the set of right-scattered points $\RR \subset \T$ is at most countable (see \cite[Lemma 3.1]{caba}).

Let $\mu_\DD$ be the Lebesgue $\DD$-measure on $\T$ defined in terms of Carath\'eodory extension in \cite[Chapter 5]{bohn3}. We also refer the reader to \cite{agar,caba,guse} for more details on the $\mu_\DD$-measure theory. In particular, for all elements $a,b$ of $\T$ such that $a \leq b $, one has $\mu_\DD ([a,b[_\T) = b-a$. Recall that $A \subset \T$ is a $\mu_\DD$-measurable set of $\T$ if and only if $A$ is an usual $\mu_L$-measurable set of $\R$, where $\mu_L$ denotes the usual Lebesgue measure (see \cite[Proposition 3.1]{caba}). Moreover, if $A \subset \T \backslash \{ \sup \T \}$, then 
$$\mu_\DD ( A ) = \mu_L (A) + \di \sum_{r \in A \cap \RR} \mu (r).$$
Let $A \subset \T$. A property is said to hold $\DD$-almost everywhere (shortly $\DD$-a.e.) on $A$ if it holds for every $t \in A \backslash A_0$, where $A_0 \subset A$ is some $\mu_\DD$-measurable subset of $\T$ satisfying $\mu_\DD (A_0) = 0$. In particular, since $\mu_\DD (\{ r \}) = \mu (r) > 0$ for every $r \in \RR$, we conclude that if a property holds $\DD$-a.e. on $A$, then it holds for every $r \in A \cap \RR$. 

Let $A \subset \T \backslash \{ \sup \T \}$ be a $\mu_\DD$-measurable set of $\T$. Consider a function $q$ defined $\DD$-a.e. on $A$ with values in $\R^n$. Let $\tilde{A}=A \cup ]r,\sigma(r)[_{r \in A \cap \RR}$, and let $\tilde{q}$ be the extension of $q$ defined $\mu_L$-a.e. on $\tilde{A}$ by
\begin{equation*}
\tilde{q} (t) = \left\lbrace \begin{array}{rcl} 
q(t) & \textrm{if} & t \in A \\ 
q(r) & \textrm{if} & t \in ]r,\sigma(r)[ \ \textrm{for every} \; r \in A \cap \RR. 
\end{array} \right.
\end{equation*}
We recall that $q$ is $\mu_\DD$-measurable on $A$ if and only if $\tilde{q}$ is $\mu_L$-measurable on $\tilde{A}$ (see \cite[Proposition 4.1]{caba}). 

The functional space $\L^\infty_\T (A,\R^n)$ is the set of all functions $q$ defined $\DD$-a.e. on $A$, with values in $\R^n$, that are $\mu_\DD$-measurable on $A$ and such that
$$
\supess\limits_{\tau \in A} \Vert q(\tau) \Vert < + \infty.
$$
Endowed with the norm $\Vert q \Vert_{\L^\infty_\T (A)} = \supess\limits_{\tau \in A} \Vert q(\tau) \Vert$, it is a Banach space (see \cite[Theorem 2.5]{agar}).
The functional space $\L^1_\T (A,\R^n)$ is the set of all functions $q$ defined $\DD$-a.e. on $A$, with values in $\R^n$, that are $\mu_\DD$-measurable on $A$ and such that
$$
\int_{A} \Vert q(\tau) \Vert \; \DD \tau < + \infty.
$$
Endowed with the norm $\Vert q \Vert_{\L^1_\T (A,\R^n)} = \int_{A} \Vert q(\tau) \Vert \; \DD \tau$, it is a Banach space (see \cite[Theorem 2.5]{agar}).
Recall that if $q \in \L^1_\T (A,\R^n)$ then
$$ \int_{A} q(\tau) \; \DD \tau = \int_{\tilde{A}} \tilde{q}(\tau) \; d\tau = \di \int_{A} q(\tau) \; d\tau +  \sum_{r \in A \cap \RR} \mu (r) q(r)$$
(see \cite[Theorems 5.1 and 5.2]{caba}). Note that if $A$ is bounded then $\L^\infty_\T (A,\R^n) \subset \L^1_\T (A,\R^n)$.

\subsection{Properties of absolutely continuous functions}\label{section13}
Let $a$ and $b$ be two elements of $\T$ such that $a < b$. Let $\CC ([a,b]_\T,\R^n)$ denote the  space of continuous functions defined on $[a,b]_\T$ with values in $\R^n$. Endowed with its usual norm $\Vert \cdot \Vert_\infty$, it is a Banach space. Let $\AC ([a,b]_\T,\R^n)$ denote the subspace of absolutely continuous functions. We recall the two following results.

\begin{proposition}\label{prop13-1}
Let $t_0 \in [a,b]_\T$ and $q:[a,b]_\T\rightarrow\R^n$. Then $q \in \AC([a,b]_\T,\R^n)$ if and only if the two following conditions are satisfied:
\begin{enumerate}
\item $q$ is $\DD$-differentiable $\DD$-a.e. on $[a,b[_\T$ and $q^\DD \in \L^1_\T ([a,b[_\T,\R^n)$;
\item For every $t \in [a,b]_\T$, there holds
$$ q(t) = q(t_0) +  \int_{[t_0,t[_\T} q^\DD (\tau) \; \DD \tau$$
whenever $t \geq t_0$, and
$$ q(t) = q(t_0) - \int_{[t,t_0[_\T} q^\DD (\tau) \; \DD \tau $$
whenever $ t \leq t_0$.
\end{enumerate}
\end{proposition}

This result can be easily derived from \cite[Theorem 4.1]{caba2}. 
By combining Proposition~\ref{prop13-1} and the usual Lebesgue's point theory in $\R$, we infer the following result (see also \cite{zhan3} for a similar result).

\begin{proposition}\label{prop13-2}
Let $t_0 \in [a,b]_\T$ and $q \in \L^1_\T ([a,b[_\T,\R^n)$. Let $Q$ be the function defined on $[a,b]_\T$ by
$$Q(t) = \int_{[t_0,t[_\T} q (\tau) \; \DD \tau$$
if $ t \geq t_0$, and by
$$ Q(t) = - \int_{[t,t_0[_\T} q (\tau) \; \DD \tau $$
if $t \leq t_0$.
Then $Q \in \AC([a,b]_\T)$ and $Q^\DD = q$ $\DD$-a.e. on $[a,b[_\T$.
\end{proposition}

\section{General non shifted $\DD$-Cauchy problem}\label{section2}
Throughout this section we consider the general non shifted $\DD$-Cauchy problem
\begin{equation*}
\mathrm{(\DD\text{-}CP)}  \begin{split}
q^\DD(t) &= f (q(t),t), \\
q(t_0)&=q_0,
\end{split} 
\end{equation*}
where $t_0 \in \T$, $q_0 \in \Omega$, where $\Omega$ is a non empty open subset of $\R^n$, and $f:\Omega \times \T \setminus \{ \sup \T \}\rightarrow\R^n$ is a $\DD$-Carath\'eodory function. The notation $\KK$ stands for the set of compact subsets of $\Omega$.

\subsection{Preliminaries}\label{section20}
In what follows it will be important to distinguish between three cases:
\begin{enumerate}
\item $t_0 = \min \T$;
\item $t_0 = \max \T$;
\item $t_0 \neq \inf \T$ and $t_0 \neq \sup \T$.
\end{enumerate}
Indeed, the interval of definition of a solution of $\mathrm{(\DD\text{-}CP)}$ will depend on the specific case under consideration. If $t_0 = \min \T$, then a solution can only \textit{go forward} since $]-\infty,t_0[_\T = \emptyset$. If $t_0 = \max \T$, then a solution can only \textit{go backward}. If $t_0 \neq \inf \T$ and $t_0 \neq \sup \T$, then a solution can go \textit{backward} and \textit{forward}. 

\begin{definition}
For all $(a,b) \in \T^2$, we say that $\ab$ if
\begin{itemize}
\item $a = t_0 < b$ in the case $t_0 = \min \T$;
\item $a < t_0 = b$ in the case $t_0 = \max \T$;
\item $a < t_0 < b$ in the case $t_0 \neq \inf \T$ and $t_0 \neq \sup \T$.
\end{itemize}
\end{definition}

If $\ab$ then $[a,b]_\T$ is a potential interval of definition for a solution of $\mathrm{(\DD\text{-}CP)}$. 
Due to this difference of intervals, it is required to make different assumptions on $f$ accordingly, whence the following series of definitions.

\begin{definition}
The function $f$ is said to be \emph{locally bounded on $\Omega \times \T \backslash \{ \sup \T \}$} if, for every $K \in \KK$, for all $(a, b) \in \T^2$ such that $ a < b$, there exists $M \geq 0$ such that
\begin{equation}\label{eqcondinfini}\tag{H${}_\infty$}
\Vert f(x,t) \Vert \leq M,
\end{equation}
for every $x\in K$ and for $\DD$-a.e. $t \in [a,b[_\T$.
\\
In what follows this property will be referred to as \eqref{eqcondinfini}.
\end{definition}

\begin{definition}
The function $f$ is said to be \emph{locally Lipschitz continuous with respect to the first variable at right-dense points} if, for every $\overline{x} \in \Omega$ and every right-dense point $\overline{t}\in\T \backslash \{ \sup \T \}$, there exist $R > 0$, $\delta > 0$ and $L \geq 0$ such that $\overline{B}(\overline{x},R) \subset \Omega$ and $\overline{t}+\delta \in \T$, and such that
\begin{equation}\label{eqcondrdloc-Lip}\tag{H${}^{\mathrm{rd}}_{\mathrm{loc-Lip}}$}
\Vert f(x_1,t) - f(x_2,t) \Vert \leq L \Vert x_1 - x_2 \Vert,
\end{equation}
for all $ x_1, x_2 \in \overline{B}(\overline{x},R)$ and for $\DD$-a.e. $t \in [\overline{t},\overline{t}+\delta[_\T$.
\\
In what follows this property will be referred to as \eqref{eqcondrdloc-Lip}.
\end{definition}

\begin{definition}
The function $f$ is said to be \emph{forward $\Omega$-stable at right-scattered points} if the mapping
\begin{equation}\label{eqcondfsta}\tag{H${}^{\mathrm{forw}}_{\mathrm{stab}}$}
\fonction{G^+(t)}{\Omega}{\R^n}{x}{x+\mu (t) f(x,t)} 
\end{equation}
takes its values in $\Omega$, for every $t\in\RR$.
\\
In what follows this property will be referred to as \eqref{eqcondfsta}.
\end{definition}

\begin{definition}
The function $f$ is said to be \emph{locally Lipschitz continuous with respect to the first variable at left-dense points} if, for every $\overline{x} \in \Omega$ and every left-dense point $\overline{t}\in\T \backslash \{ \inf \T \}$, there exist $R > 0$, $\delta > 0$ and $L \geq 0$ such that $\overline{B}(\overline{x},R) \subset \Omega$ and $\overline{t}-\delta \in \T$ and such that
\begin{equation}\label{eqcondldloc-Lip}\tag{H${}^{\mathrm{ld}}_{\mathrm{loc-Lip}}$}
\Vert f(x_1,t) - f(x_2,t) \Vert \leq L \Vert x_1 - x_2 \Vert,
\end{equation}
for all $x_1, x_2 \in \overline{B}(\overline{x},R)$ and for $\DD$-a.e. $t \in [\overline{t}-\delta,\overline{t}[_\T$.
\\
In what follows this property will be referred to as \eqref{eqcondldloc-Lip}.
\end{definition}

\begin{definition}
The function $f$ is said to be \emph{backward regressive at right-scattered points} if
\begin{equation}\label{eqcondbreg}\tag{H${}^{\mathrm{back}}_{\mathrm{regr}}$}
G^+(t) \; \text{is invertible},
\end{equation}
for every $t\in\RR$.
\\
In what follows this property will be referred to as \eqref{eqcondbreg}.
\end{definition}

Assumption \eqref{eqcondinfini} will be instrumental to provide a $\DD$-integral characterization of the solutions of $\mathrm{(\DD\text{-}CP)}$ (see Lemma \ref{prop21-1} in Section \ref{app11}). The other assumptions play a role in order to \textit{go forward} or \textit{backward} for a solution of a non shifted $\DD$-Cauchy problem. More precisely, \eqref{eqcondrdloc-Lip} and \eqref{eqcondfsta} allow to go forward, and \eqref{eqcondldloc-Lip} and \eqref{eqcondbreg} allow to go backward (see the proofs of Propositions \ref{prop21-3} and \ref{prop21-4} in Section \ref{app11} for more details).

In view of investigating global solutions, the following definition will be also useful.

\begin{definition}
The function $f$ is said to be \emph{globally Lipschitz continuous} if there exists $L\geq 0$ such that
\begin{equation}\label{eqcondgloblip}\tag{H$^{\mathrm{glob}}_{\mathrm{Lip}}$}
\Vert f(x_1,t) - f(x_2,t) \Vert \leq L \Vert x_1 - x_2 \Vert.
\end{equation}
for all $x_1,x_2\in\Omega$ and for $\DD$-a.e. $t \in \T \backslash \{ \sup \T \}$.
\\
In what follows this property will be referred to as \eqref{eqcondgloblip}.
\end{definition}

\subsection{Definition of a maximal solution}\label{section21}
We first define the notion of a solution of $\mathrm{(\DD\text{-}CP)}$ on an interval $[a,b]_\T$ with $\ab$.

\begin{definition}
Let $(a, b) \in \T^2$ be such that $\ab$ and let $q:[a,b]_\T\rightarrow\Omega$. The couple $(q,[a,b]_\T)$ is said to be a solution of $\mathrm{(\DD\text{-}CP)}$ if $q \in \AC([a,b]_\T)$, if $q(t_0) = q_0$, and if $q^\DD (t) = f(q(t),t)$ for $\DD$-a.e. $t \in [a,b[_\T$.
\end{definition}

Note that, if $(q,[a,b]_\T)$ is a solution of $\mathrm{(\DD\text{-}CP)}$, then $(q,[a',b']_\T)$ is as well a solution of $\mathrm{(\DD\text{-}CP)}$ for all $a',b' \in [a,b]_\T$ satisfying $a' \trianglelefteq t_0 \trianglelefteq b'$. 

In view of defining the notion of a solution of $\mathrm{(\DD\text{-}CP)}$ on more general intervals, we set
$$\I = \{ I_\T\ \vert\ \exists a, b \in I_\T, \; \ab \}.$$
The set $\I$ is the set of potential intervals of $\T$ for a solution of  $\mathrm{(\DD\text{-}CP)}$. 

\begin{definition}
Let $I_\T \in \I$ and let $q:I_\T\rightarrow\Omega$. The couple $(q,I_\T)$ is said to be a solution of $\mathrm{(\DD\text{-}CP)}$ if $(q,[a,b]_\T)$ is a solution of $\mathrm{(\DD\text{-}CP)}$ for all $a,b \in I_\T$ satisfying $\ab$.
\end{definition}

Finally, we define the concept of a maximal solution.

\begin{definition}
Let $(q,I_\T)$ and $(q_1,I^1_\T)$ be two solutions of $\mathrm{(\DD\text{-}CP)}$. The solution $(q_1,I^1_\T)$ is said to be an extension of the solution $(q,I_\T)$ if $I_\T \subset I^1_\T$ and $q_1 = q$ on $I_\T$.
A solution $(q,I_\T)$ of $\mathrm{(\DD\text{-}CP)}$ is said to be maximal if, for every extension $(q_1,I^1_\T)$ of $(q,I_\T)$, there holds $I^1_\T = I_\T$.
A solution $(q,I_\T)$ of $\mathrm{(\DD\text{-}CP)}$ is said to be global if $I_\T= \T$.
\end{definition}

Note that, if $(q,I_\T)$ is a global solution of $\mathrm{(\DD\text{-}CP)}$, then $(q,I_\T)$ is a maximal solution of $\mathrm{(\DD\text{-}CP)}$.

\subsection{Main results}\label{section22}
Recall that we consider the general non shifted $\DD$-Cauchy problem
\begin{equation*}
\mathrm{(\DD\text{-}CP)}  \begin{split}
q^\DD(t) &= f (q(t),t), \\
q(t_0)&=q_0,
\end{split} 
\end{equation*}
where $t_0 \in \T$, $q_0 \in \Omega$, where $\Omega$ is a non empty open subset of $\R^n$, and $f:\Omega \times \T \setminus \{ \sup \T \}\rightarrow\R^n$ is a $\DD$-Carath\'eodory function.
We have the following general Cauchy-Lipschitz result.

\begin{theorem}\label{thm21-1}
We make the following assumptions on the dynamics $f$, depending on $t_0$.

\begin{enumerate}
\item If $t_0 = \min \T$, then we assume that
\begin{itemize}
\item $f$ satisfies \eqref{eqcondinfini}, that is, $f$ is locally bounded on $\Omega \times \T \backslash \{ \sup \T \}$;
\item $f$ satisfies \eqref{eqcondrdloc-Lip}, that is, $f$ is locally Lipschitz continuous with respect to the first variable at right-dense points;
\item $f$ satisfies \eqref{eqcondfsta}, that is, $f$ is forward $\Omega$-stable at right-scattered points.
\end{itemize}

\item If $t_0 = \max \T$, then we assume that
\begin{itemize}
\item $f$ satisfies \eqref{eqcondinfini}, that is, $f$ is locally bounded on $\Omega \times \T \backslash \{ \sup \T \}$;
\item $f$ satisfies \eqref{eqcondldloc-Lip}, that is, $f$ is locally Lipschitz continuous with respect to the first variable at left-dense points;
\item $f$ satisfies \eqref{eqcondbreg}, that is, $f$ is backward regressive in right-scattered points.
\end{itemize}

\item If $t_0 \neq \inf \T$ and $t_0 \neq \sup \T$, then we assume that
\begin{itemize}
\item $f$ satisfies \eqref{eqcondinfini}, that is, $f$ is locally bounded on $\Omega \times \T \backslash \{ \sup \T \}$;
\item $f$ satisfies \eqref{eqcondrdloc-Lip}, that is, $f$ is locally Lipschitz continuous with respect to the first variable at right-dense points;
\item $f$ satisfies \eqref{eqcondfsta}, that is, $f$ is forward $\Omega$-stable at right-scattered points;
\item $f$ satisfies \eqref{eqcondldloc-Lip}, that is, $f$ is locally Lipschitz continuous with respect to the first variable at left-dense points;
\item $f$ satisfies \eqref{eqcondbreg}, that is, $f$ is and backward regressive in right-scattered points.
\end{itemize}
\end{enumerate}
Then, the non shifted $\DD$-Cauchy problem $\mathrm{(\DD\text{-}CP)}$ has a unique maximal solution $(q,I_\T)$. Moreover, $(q,I_\T)$ is the maximal extension of any other solution of $\mathrm{(\DD\text{-}CP)}$.
\end{theorem}

This theorem is proved in Section \ref{app11}.
The following result gives information on the behavior of a maximal solution at its terminal points.

\begin{theorem}\label{thm21-2}
Under the assumptions of Theorem~\ref{thm21-1}, let $(q,I_\T)$ be the maximal solution of the non shifted $\DD$-Cauchy problem $\mathrm{(\DD\text{-}CP)}$. Then either $I_\T = \T$, that is, the maximal solution $(q,I_\T)$ is global, or the maximal solution is not global and then
\begin{enumerate}
\item if $t_0 = \min \T$ then $I_\T = [t_0,b[_\T$ where $b \in ]t_0,+\infty[_\T$ is a left-dense point of $\T$;
\item if $t_0 = \max \T$ then $I_\T = ]a,t_0]_\T$ where $a \in ]-\infty,t_0[_\T$ is a right-dense point of $\T$;
\item if $t_0 \neq \inf \T$ and $t_0 \neq \sup \T$ then $I_\T = ]a,+\infty[_\T$ or $I_\T = ]-\infty,b[_\T$ or $I_\T = ]a,b[_\T$, where $a \in ]-\infty,t_0[_\T$ is a right-dense point of $\T$ and $b \in ]t_0,+\infty[_\T$ is a left-dense point of $\T$;
\end{enumerate}
and moreover, for every $K \in \KK$ there exists $t\in I_\T$ (close to $a$ or $b$ depending on the cases listed above) such that $q(t)\in\Omega\setminus K$.
\end{theorem}

This theorem is proved in Section \ref{app12}. It states that the maximal solution must go out of any compact of $\Omega$ near its terminal points whenever it is not global.

The following last result states that, under global Lipschitz assumption, the maximal solution is global.

\begin{theorem}\label{thm22-1}
If $t_0 = \min \T$, $\Omega= \R^n$, if $f$ satisfies \eqref{eqcondinfini}, that is, $f$ is locally bounded on $\R^n \times \T \backslash \{ \sup \T \}$, and if $f$ satisfies \eqref{eqcondgloblip}, that is, $f$ is globally Lipschitz continuous, then the non shifted $\DD$-Cauchy problem $\mathrm{(\DD\text{-}CP)}$ has a unique maximal solution $(q,I_\T)$, which is moreover global.
\end{theorem}

The proof is done in Section \ref{app13}.

\begin{remark}\label{rem1}
As an application of Theorem \ref{thm22-1}, we recover the well known fact that, in the linear case
$$q^\DD(t) = h(t) \times q(t),$$
where $h \in \L^\infty_\T (\T \backslash \{ \sup \T \},\R^{n\times n})$, solutions are global.
\end{remark}

\subsection{Further comments}\label{section24}
In this section, we provide simple examples (in the one-dimensional case) showing the sharpness of the assumptions made in Theorem~\ref{thm21-1}. Indeed, if one of these assumptions is not satisfied, then the existence or the uniqueness of the maximal solution is no more guaranteed.

\begin{example}[Lack of Assumption~\eqref{eqcondrdloc-Lip} in the first case]
Let $\T = [0,+\infty[$, $\Omega = \R$, $t_0 = 0$, $q_0 = 0$ and $f:\R\times\T\rightarrow\R$ be defined by $f(x,t)=2 \sqrt{\vert x \vert}$.
The function $f$ obviously satisfies \eqref{eqcondfsta} since $\RR = \emptyset$, however it does not satisfy \eqref{eqcondrdloc-Lip}. The corresponding $\DD$-Cauchy problem $\mathrm{(\DD\text{-}CP)}$ has two global solutions $q_1$ and $q_2$ given by $q_1(t) = 0$ and $q_2(t) = t^2$, for every $t\in\T$.

This example shows that, in the absence of Assumption~\eqref{eqcondrdloc-Lip}, the uniqueness of the maximal solution is not guaranteed.
\end{example}

\begin{example}[Lack of Assumption~\eqref{eqcondfsta} in the first case]
Let $\T = \{ 0,1 \}$, $\Omega = ]-1,1[$, $t_0 = 0$, $q_0 = 0$ and $f:\Omega\times\{0\}\rightarrow\R$ be defined by $f(x,t)=1$.
The function $f$ obviously satisfies \eqref{eqcondrdloc-Lip} since $\T \backslash \{ \sup \T \} = \{ 0 \}$ does not admit any right-dense point of $\T$, however it does not satisfy \eqref{eqcondfsta} since $x+1 \notin \Omega$ for $x \in [0,1[$. Since $q(0) = 0$ and $q(1)=q(0)+\mu(0) f(q(0),0)$ imply $q(1) = 1 \notin \Omega$, we conclude that $\mathrm{(\DD\text{-}CP)}$ does not admit any solution.

Therefore, in the absence of Assumption~\eqref{eqcondfsta}, $\mathrm{(\DD\text{-}CP)}$ may fail to have a solution.
\end{example}

\begin{example}[Lack of Assumption~\eqref{eqcondldloc-Lip} in the second case]
Let $\T = ]-\infty,0]$, $\Omega = \R$, $t_0 = 0$, $q_0 = 0$ and $f:\R\times\T\rightarrow\R$ be defined by $f(x,t)=-2 \sqrt{\vert x \vert}$.
The function $f$ obviously satisfies \eqref{eqcondbreg} since $\RR = \emptyset$, however it does not satisfy \eqref{eqcondldloc-Lip}. The corresponding $\DD$-Cauchy problem $\mathrm{(\DD\text{-}CP)}$ ha two global solutions $q_1$ and $q_2$ given by $q_1(t) = 0$ and $q_2(t) = t^2$ for every $t\in\T$.

This example shows that, in the absence of Assumption~\eqref{eqcondldloc-Lip}, the uniqueness of the maximal solution is not guaranteed.
\end{example}

\begin{example}[Lack of Assumption~\eqref{eqcondbreg} in the second case]
Let $\T = \{ 0,1 \}$, $\Omega = \R$, $t_0 = 1$, $q_0 \in \R$ and $f:\R \times \{ 0 \}\rightarrow\R$ be defined by $f(x,t)=-x$.
The function $f$ obviously satisfies \eqref{eqcondldloc-Lip} since $\T \backslash \{ \inf \T \} = \{ 1 \}$ does not admit any left-dense point of $\T$, however it does not satisfy \eqref{eqcondbreg} since $G^+(0) = 0$. As a consequence, if $q_0 \neq 0$, $\mathrm{(\DD\text{-}CP)}$ does not admit any solution. Indeed, $q(1)=q_0$ and $q(1)=q(0)+\mu (0) f(q(0),0)$ imply $q(1)=0$, which is a contradiction. If $q_0 = 0$, we obtain an infinite number of global solutions. Indeed, any function $q$ defined on $\T$ with $q(1) = 0$ is then a global solution of $\mathrm{(\DD\text{-}CP)}$.
\end{example}

\section{General shifted $\DD$-Cauchy problem}\label{section3}
Throughout this section we consider the general \textit{shifted} $\DD$-Cauchy problem
\begin{equation*}
\mathrm{(\DD\text{-}CP^\sigma)} \begin{split}
q^\DD(t) & = f (q^\sigma(t),t), \\
q(t_0)&=q_0,
\end{split} 
\end{equation*}
where $t_0 \in \T$, $q_0 \in \Omega$, where $\Omega$ is a non empty open subset of $\R^n$ and $f:\Omega \times \T \backslash \{ \sup \T \}\rightarrow \R^n$ is a $\DD$-Carath\'eodory function. 

The results of the section follow the same lines as in the previous section. Therefore we do not give any proof nor counterexamples as above. Some comments are however done in Section \ref{app2}. 

\subsection{Preliminaries}\label{section30}
As in Section~\ref{section20}, it will be important to distinguish between three cases:
\begin{enumerate}
\item $t_0 = \min \T$;
\item $t_0 = \max \T$;
\item $t_0 \neq \inf \T$ and $t_0 \neq \sup \T$.
\end{enumerate}

With respect to Section \ref{section20}, we introduce two additional concepts.

\begin{definition}
The function $f$ is said to be \emph{backward $\Omega$-stable at right-scattered points} if the mapping
\begin{equation}\label{eqcondbsta}\tag{H${}^{\mathrm{back}}_{\mathrm{stab}}$}
\fonction{G^-(t)}{\Omega}{\R^n}{x}{x-\mu (t) f(x,t)} 
\end{equation}
takes its values in $\Omega$, for every $t \in \RR$.
\\
In what follows this property will be referred to as \eqref{eqcondbsta}.
\end{definition}

\begin{definition}
The function $f$ is said to be \emph{forward regressive at right-scattered points} if
\begin{equation}\label{eqcondfreg}\tag{H${}^{\mathrm{forw}}_{\mathrm{regr}}$}
\fonctionsansdef{G^-(t)}{\Omega}{\R^n} \; \text{is invertible},
\end{equation}
for every $t \in \RR$.
\\
In what follows this property will be referred to as \eqref{eqcondfreg}.
\end{definition}

These above assumptions play a role in order to \textit{go forward} or \textit{backward} for a solution of a shifted $\DD$-Cauchy problem. Precisely, \eqref{eqcondrdloc-Lip} and \eqref{eqcondfreg} allow to \textit{go forward}. Similarly, \eqref{eqcondldloc-Lip} and \eqref{eqcondbsta} allow to \textit{go backward}. \\

\subsection{Definition of a maximal solution}\label{section31}

\begin{definition}
Let $(a,b) \in \T^2$ satisfying $\ab$ and let $\fonctionsansdef{q}{[a,b]_\T}{\Omega}$. The couple $(q,[a,b]_\T)$ is said to be a solution of $\mathrm{(\DD\text{-}CP^\sigma)}$ if $q \in \AC([a,b]_\T)$, $q(t_0) = q_0$, and $q^\DD (t) = f(q^\sigma(t),t)$ for $\DD$-a.e. $t \in [a,b[_\T$.
\end{definition}

\begin{definition}
Let $I_\T \in \I$ and let $\fonctionsansdef{q}{I_\T}{\Omega}$. The couple $(q,I_\T)$ is said to be a solution of $\mathrm{(\DD\text{-}CP^\sigma)}$ if $(q,[a,b]_\T)$ is a solution of $\mathrm{(\DD\text{-}CP^\sigma)}$ for all $a,b \in I_\T$ satisfying $\ab$.
\end{definition}

\begin{definition}
Let $(q,I_\T)$ and $(q_1,I^1_\T)$ be two solutions of $\mathrm{(\DD\text{-}CP^\sigma)}$. The solution $(q_1,I^1_\T)$ is said to be an extension of the solution $(q,I_\T)$ if $I_\T \subset I^1_\T $ and $q_1 = q$ on $I_\T$.
A solution $(q,I_\T)$ of $\mathrm{(\DD\text{-}CP^\sigma)}$ is said to be \emph{maximal} if, for every extension $(q_1,I^1_\T)$ of $(q,I_\T)$, there holds $I^1_\T = I_\T$.
A solution $(q,I_\T)$ of $\mathrm{(\DD\text{-}CP^\sigma)}$ is said to be \emph{global} if $I_\T= \T$.
\end{definition}

\subsection{Main results}\label{section32}
Recall that we consider the general shifted $\DD$-Cauchy problem
\begin{equation*}
\mathrm{(\DD\text{-}CP^\sigma)} \begin{split}
q^\DD(t) & = f (q^\sigma(t),t), \\
q(t_0)&=q_0,
\end{split}
\end{equation*}
where $t_0 \in \T$, $q_0 \in \Omega$ where $\Omega$ is a non empty open subset of $\R^n$ and $f:\Omega \times \T \backslash \{ \sup \T \}\rightarrow \R^n$ is a  $\DD$-Carath\'eodory function.

\begin{theorem}\label{thm31-1}
We make the following assumptions on the dynamics $f$, depending on $t_0$.
\begin{enumerate}
\item If $t_0 = \min \T$, then we assume that
\begin{itemize}
\item $f$ satisfies \eqref{eqcondinfini}, that is, $f$ is locally bounded on $\Omega \times \T \backslash \{ \sup \T \}$;
\item $f$ satisfies \eqref{eqcondrdloc-Lip}, that is, $f$ is locally Lipschitz continuous with respect to the first variable at right-dense points;
\item $f$ satisfies \eqref{eqcondfreg}, that is, $f$ is forward regressive in right-scattered points.
\end{itemize}

\item If $t_0 = \max \T$, then we assume that
\begin{itemize}
\item $f$ satisfies \eqref{eqcondinfini}, that is, $f$ is locally bounded on $\Omega \times \T \backslash \{ \sup \T \}$;
\item $f$ satisfies \eqref{eqcondldloc-Lip}, that is, $f$ is locally Lipschitz continuous with respect to the first variable at left-dense points;
\item $f$ satisfies \eqref{eqcondbsta}, that is, $f$ is backward $\Omega$-stable in right-scattered points.
\end{itemize}

\item If $t_0 \neq \inf \T$ and $t_0 \neq \sup \T$, then we assume that
\begin{itemize}
\item $f$ satisfies \eqref{eqcondinfini}, that is, $f$ is locally bounded on $\Omega \times \T \backslash \{ \sup \T \}$;
\item $f$ satisfies \eqref{eqcondrdloc-Lip}, that is, $f$ is locally Lipschitz continuous with respect to the first variable at right-dense points;
\item $f$ satisfies \eqref{eqcondfreg}, that is, $f$ is forward regressive at right-scattered points;
\item $f$ satisfies \eqref{eqcondldloc-Lip}, that is, $f$ is locally Lipschitz continuous with respect to the first variable at left-dense points;
\item $f$ satisfies \eqref{eqcondbsta}, that is, $f$ is backward $\Omega$-stable at right-scattered points.
\end{itemize}
\end{enumerate}
Then the shifted $\DD$-Cauchy problem $\mathrm{(\DD\text{-}CP^\sigma)}$ has a unique maximal solution $(q,I_\T)$. Moreover, $(q,I_\T)$ is the maximal extension of any other solution of $\mathrm{(\DD\text{-}CP^\sigma)}$
\end{theorem}

\begin{theorem}\label{thm31-2}
Under the assumptions of Theorem~\ref{thm31-1}, let $(q,I_\T)$ be the maximal solution of the shifted $\DD$-Cauchy problem $\mathrm{(\DD\text{-}CP^\sigma)}$. Then either $I_\T = \T$, that is, the maximal solution $(q,I_\T)$ is global, or the maximal solution is not global and then
\begin{enumerate}
\item if $t_0 = \min \T$ then $I_\T = [t_0,b[_\T$ where $b \in ]t_0,+\infty[_\T$ is a left-dense point of $\T$;
\item if $t_0 = \max \T$ then $I_\T = ]a,t_0]_\T$ where $a \in ]-\infty,t_0[_\T$ is a right-dense point of $\T$;
\item if $t_0 \neq \inf \T$ and $t_0 \neq \sup \T$ then $I_\T = ]a,+\infty[_\T$ or $I_\T = ]-\infty,b[_\T$ or $I_\T = ]a,b[_\T$ where $a \in ]-\infty,t_0[_\T$ is a right-dense point of $\T$ and $b \in ]t_0,+\infty[_\T$ is a left-dense point of $\T$;
\end{enumerate}
and moreover, for every $K \in \KK$ there exists $t\in I_\T$ (close to $a$ or $b$ depending on the cases listed above) such that $q(t)\in\Omega\setminus K$.
\end{theorem}

\begin{theorem}\label{thm32-1}
If $t_0 = \max \T$, $\Omega= \R^n$, if $f$ satisfies \eqref{eqcondinfini}, that is, $f$ is locally bounded on $\R^n \times \T \backslash \{ \sup \T \}$, and if $f$ satisfies \eqref{eqcondgloblip}, that is, $f$ is globally Lipschitz continuous, then, the shifted $\DD$-Cauchy problem $\mathrm{(\DD\text{-}CP^\sigma)}$ has a unique maximal solution $(q,I_\T)$, which is moreover global.
\end{theorem}

\begin{remark}
As in Remark \ref{rem1}, in the linear case the maximal solution of any shifted $\DD$-Cauchy problem is automatically global.
\end{remark}

\section{Proofs of the results}\label{app1}
\subsection{Proof of Theorem~\ref{thm21-1}}\label{app11}
If $f$ satisfies \eqref{eqcondinfini}, then for all $(a,b)\in\T^2$ such that $a<b$, there holds
\begin{equation}\label{eq666}
f(q,t) \in \L^\infty_\T ([a,b[_\T,\R^n) \subset \L^1_\T ([a,b[_\T,\R^n),
\end{equation}
for every $q \in \CC ([a,b]_\T,\R^n)$. Then, from Section \ref{section13}, we have the following $\DD$-integral characterization of the solutions of $\mathrm{(\DD\text{-}CP)}$.

\begin{lemma}\label{prop21-1}
Let $I_\T \in \I$ and let $q:I_\T\rightarrow\Omega$. If $f$ satisfies \eqref{eqcondinfini}, then the couple $(q,I_\T)$ is a solution of $\mathrm{(\DD\text{-}CP)}$ if and only if for all $a,b \in I_\T$ satisfying $\ab$, one has $q \in \CC([a,b]_\T)$ and
\begin{equation*}
q(t) = \left\lbrace \begin{array}{lcc}
q_0 + \int_{[t_0,t[_\T} f(q (\tau),\tau) \; \DD \tau & \text{if} & t \geq t_0, \\
q_0 - \int_{[t,t_0[_\T} f(q (\tau),\tau) \; \DD \tau & \text{if} & t \leq t_0. \\
\end{array} \right.
\end{equation*}
for every $ t \in [a,b]_\T$.
\end{lemma}

This characterization allows one to prove the following result.

\begin{lemma}\label{prop21-2}
If $f$ satisfies \eqref{eqcondinfini}, then every solution of $\mathrm{(\DD\text{-}CP)}$ can be extended to a maximal solution.
\end{lemma}

\begin{proof}
Let $(q,I_\T)$ be a solution of $\mathrm{(\DD\text{-}CP)}$. Let us define the non empty set $\mathscr{F}$ of extensions of $(q,I_\T)$. The set $\mathscr{F}$ is ordered by
\begin{equation*}
(q_1,I^1_\T) \leq (q_2,I^2_\T) \; \text{if and only if} \; (q_2,I^2_\T) \; \text{is an extension of} \; (q_1,I^1_\T).
\end{equation*}
Let us prove that $\mathscr{F}$ is inductive. Let $\GG = \cup_{p \in \mathscr{P}} \{ (q_p,I^p_\T) \}$ be a non empty totally ordered subset of $\mathscr{F}$. Let us prove that $\GG$ admits an upper bound.

Let us define $\overline{I} = \cup_{p \in \mathscr{P}} I^p $. This is an interval of $\R$, since $t_0 \in \cap_{p \in \mathscr{P}} I^p$. Then $\overline{I}_\T = \cup_{p \in \mathscr{P}} I^p_\T \in \I$. For every $t \in \overline{I}_\T$, there exists $p \in \mathscr{P}$ such that $t \in I^p_\T$ and, since $\GG$ is totally ordered, if $t \in I^{p_1}_\T \cap I^{p_2}_\T$ then $q_{p_1} (t) = q_{p_2} (t)$. Consequently, we can define $\overline{q}$ by
\begin{equation}
\forall t \in \overline{I}_\T, \; \overline{q}(t) = q_p (t) \in \Omega \; \text{where} \; t \in I^p_\T.
\end{equation}
Our aim is to prove that $(\overline{q},\overline{I}_\T)$ is a solution of $\mathrm{(\DD\text{-}CP)}$. Let $a,b \in \overline{I}_\T$ satisfying $\ab$. Since $\GG$ is totally ordered, there exists $p \in \mathscr{P}$ such that $[a,b]_\T \subset I^p_\T$ and $\overline{q} = q_p$ on $[a,b]_\T$. Since $(q_p,I^p_\T) $ is a solution of $\mathrm{(\DD\text{-}CP)}$, we obtain that $q_p$ satisfies the necessary and sufficient condition of Lemma~\ref{prop21-1} on $[a,b]_\T$. Consequently, this holds true as well for $\overline{q}$ on $[a,b]_\T$. Finally, since this last sentence is true for all $a, b \in \overline{I}_\T$ satisfying $\ab$, we infer from Lemma~\ref{prop21-1} that $(\overline{q},\overline{I}_\T)$ is a solution of $\mathrm{(\DD\text{-}CP)}$. Since $(\overline{q},\overline{I}_\T)$ is obviously an extension of any element of $\GG$, we obtain that $\GG$ admits an upper bound and then, $\mathscr{F}$ is inductive. 

Finally, $\mathscr{F}$ is a non empty ordered inductive set and consequently, from Zorn's lemma, admits a maximal element. The proof is complete.
\end{proof}

\begin{proposition}[Existence of a local solution]\label{prop21-3}
There exist $a, b \in \T$ satisfying $\ab$ and $q : [a,b]_\T\rightarrow\Omega$ such that $(q,[a,b]_\T) $ is a solution of $\mathrm{(\DD\text{-}CP)}$.
\end{proposition}

\begin{proof}
We only prove this proposition in the third case of Theorem~\ref{thm21-1} (the two first cases are derived similarly) for which $t_0 \neq \inf \T$ and $t_0 \neq \sup \T$. We distinguish between four situations.

\paragraph{First case:} $t_0$ is a left- and a right-scattered point of $\T$.
In this case, it is sufficient to consider $a= \rho (t_0) \in ]-\infty,t_0[_\T$, $b=\sigma(t_0) \in ]t_0,+\infty[_\T$ and the function $q$ defined on $[a,b]_\T = \{a, t_0,b \}$ with values in $\Omega$ by $q(a) = G^+(a)^{-1}(q_0)$, $q(t_0) = q_0$ and $q(b) = G^+(t_0)(q_0)$.
We note that $q(a)$ is well-defined in $\Omega$ from \eqref{eqcondbreg} and $q(b) \in \Omega$ from \eqref{eqcondfsta}. 

\paragraph{Second case:} $t_0$ is a left- and a right-dense point of $\T$.
Let $R'$, $\delta'$ and $L'$ associated with $q_0$ and $t_0$ in \eqref{eqcondldloc-Lip} and let $R''$, $\delta''$ and $L''$ associated with $q_0$ and $t_0$ in \eqref{eqcondrdloc-Lip}. We define $R = \min (R',R'') > 0$ and $L = \max (L',L'') \geq 0$. Let $M$ associated with $\overline{B}(q_0,R) \in \KK$ and $[t_0-\delta',t_0+\delta''[_\T$ in \eqref{eqcondinfini}. Consider $0 <\delta_1 \leq \delta'$ and $0 <\delta_2 \leq \delta''$ such that $a=t_0-\delta_1 \in ]-\infty,t_0[_\T$, $b=t_0+\delta_2 \in ]t_0,+\infty[_\T$ and $\delta_1$ and $\delta_2$ are sufficiently small in order to have $\max(\delta_1,\delta_2) M \leq R$ and $\max(\delta_1,\delta_2) L < 1$. Then, we can construct the $\max(\delta_1,\delta_2)L$-contraction map with respect to the norm $\Vert \cdot \Vert_{\infty}$
\begin{equation*}
\fonction{F}{\CC ([a,b]_\T,\overline{B}(q_0,R))}{\CC ([a,b]_\T,\overline{B}(q_0,R))}{q}{F(q),}
\end{equation*}
with
\begin{equation*}
\fonction{F(q)}{[a,b]_\T}{\overline{B}(q_0,R)}{t}{\left\lbrace \begin{array}{lcc}
q_0 + \int_{[t_0,t[_\T} f(q (\tau),\tau) \; \DD \tau & \text{if} & t \geq t_0 \\
q_0 -\int_{[t,t_0[_\T} f(q (\tau),\tau) \; \DD \tau & \text{if} & t \leq t_0. \\
\end{array} \right. }
\end{equation*}
It follows from the Banach fixed point theorem that $F$ has a unique fixed point denoted by $q$, and then $(q,[a,b]_\T)$ is a solution of $\mathrm{(\DD\text{-}CP)}$. 

\paragraph{Third case:} $t_0$ is a left-scattered and a right-dense point of $\T$.
Let $R$, $\delta$ and $L$ associated with $q_0$ and $t_0$ in \eqref{eqcondrdloc-Lip}. Let $M$ associated with $\overline{B}(q_0,R) \in \KK$ and $[t_0,t_0+\delta[_\T$ in \eqref{eqcondinfini}. Consider $0 <\delta_1 \leq \delta$ such that $b=t_0+\delta_1 \in ]t_0,+\infty[_\T$ and $\delta_1$ is sufficiently small in order to have $\delta_1 M \leq R$ and $\delta_1 L < 1$. Then, we can construct the $\delta_1 L$-contraction map with respect to the norm $\Vert \cdot \Vert_{\infty}$
\begin{equation*}
\fonction{F}{\CC ([t_0,b]_\T,\overline{B}(q_0,R))}{\CC ([t_0,b]_\T,\overline{B}(q_0,R))}{q}{F(q)}
\end{equation*}
with
$$
{\fonction{F(q)}{[t_0,b]_\T}{\overline{B}(q_0,R)}{t}{q_0 + \di \int_{[t_0,t[_\T} f(q(\tau),\tau) \; \DD \tau.}}
$$
It follows from the Banach fixed point theorem that $F$ has a unique fixed point denoted by $q$ defined on $[t_0,b]_\T$. Finally, since $t_0$ is a left-scattered point of $\T$ and from \eqref{eqcondbreg}, we define $a= \rho (t_0) \in ]-\infty,t_0[_\T$ and $q(a) = G^+(a)^{-1}(q_0) \in \Omega$. We have thus obtained a solution $(q,[a,b]_\T)$ of $\mathrm{(\DD\text{-}CP)}$. 

\paragraph{Fourth case:} $t_0$ is a left-dense and a right-scattered point of $\T$.
Let $R$, $\delta$ and $L$ associated with $q_0$ and $t_0$ in \eqref{eqcondldloc-Lip}. Let $M$ associated with $\overline{B}(q_0,R) \in \KK$ and $[t_0-\delta,t_0[_\T$ in \eqref{eqcondinfini}. Consider $0 <\delta_1 \leq \delta$ such that $a=t_0-\delta_1 \in ]-\infty,t_0[_\T$ and $\delta_1$ is sufficiently small in order to have $\delta_1 M \leq R$ and $\delta_1 L < 1$. Then, we can construct the $\delta_1 L$-contraction map with respect to the norm $\Vert \cdot \Vert_{\infty}$
\begin{equation*}
\fonction{F}{\CC ([a,t_0]_\T,\overline{B}(q_0,R))}{\CC ([a,t_0]_\T,\overline{B}(q_0,R))}{q}{F(q)}
\end{equation*}
with
$$
{\fonction{F(q)}{[a,t_0]_\T}{\overline{B}(q_0,R)}{t}{q_0 - \di \int_{[t,t_0[_\T} f(q(\tau),\tau) \; \DD \tau.}}
$$
It follows from the Banach fixed point theorem that $F$ admits a unique fixed point denoted by $q$ defined on $[a,t_0]_\T$. Since $t_0$ is a right-scattered point of $\T$, and from \eqref{eqcondfsta}, we define $b= \sigma (t_0) \in ]t_0,+\infty[_\T$ and $q(b) = G^+(t_0)(q_0) \in \Omega$. We have thus obtained a solution $(q,[a,b]_\T)$ of $\mathrm{(\DD\text{-}CP)}$. 
\end{proof}

From Lemma~\ref{prop21-2}, we can extend the solution given in Proposition~\ref{prop21-3} and we obtain the existence of a maximal solution. The following result proves that it is unique.

\begin{proposition}[Local uniqueness of a solution]\label{prop21-4}
Let $(q_1,I^1_\T)$ and $(q_2,I^2_\T)$ be two solutions of $\mathrm{(\DD\text{-}CP)}$. Then, $q_1 = q_2$ on $I^1_\T \cap I^2_\T$.
\end{proposition}

\begin{proof}
As before, we only prove this proposition in the third case of Theorem~\ref{thm21-1}.
We denote by $I = I^1 \cap I^2$ (interval of $\R$). One can easily prove that $I_\T = I^1_\T \cap I^2_\T \in \I$. It is sufficient to prove $q_1 = q_2$ on $[a,b]_\T$ for all $a,b \in I_\T$ satisfying $\ab$. Let $a$, $b \in I_\T$ satisfying $\ab$. Set 
$$
A =  \{ t \in [a,t_0]_\T, \; q_1 (t) \neq q_2 (t) \},$$
and
$$
B = \{ t \in [t_0,b]_\T, \; q_1 (t) \neq q_2 (t) \}.
$$
Let us prove by contradiction that $A \cup B = \emptyset$. Assume that $A \neq \emptyset$ and let $\overline{t} = \sup A$. Note that $\overline{t} \in [a,t_0]_\T$ (since $\T$ is closed) and that $q_1 = q_2$ on $]\overline{t},t_0]_\T$. In order to raise a contradiction, we first derive the four following facts.
\begin{enumerate}
\item \textit{Fact 1: $\overline{t} < t_0$.} If $t_0$ is a left-scattered point of $\T$, this claim is obvious since $q_1(t_0) = q_2(t_0) = q_0$ and $q_1 (\rho(t_0)) = q_2 (\rho(t_0)) = G^+(\rho(t_0))^{-1}(q_0)$ from \eqref{eqcondbreg}. If $t_0$ is a left-dense point of $\T$, let $R$, $\delta$ and $L$ associated with $q_0$ and $t_0$ in \eqref{eqcondldloc-Lip}. Let $M$ associated with $\overline{B}(q_0,R) \in \KK$ and $[t_0-\delta,t_0[_\T$ in \eqref{eqcondinfini}. Consider $0 <\delta_1 \leq \delta$ such that $c=t_0-\delta_1 \in [a,t_0[_\T$ and $\delta_1$ is sufficiently small in order to have $\delta_1 M \leq R$, $\delta_1 L < 1$ and $q_1$, $q_2 \in \CC ([c,t_0]_\T,\overline{B}(q_0,R))$. Since $q_1$ and $q_2$ are solutions of $\mathrm{(\DD\text{-}CP)}$ on $[a,b]_\T$, they are in particular fixed points of the $\delta_1 L$-contraction map
\begin{equation*}
\fonction{F}{\CC ([c,t_0]_\T,\overline{B}(q_0,R))}{\CC ([c,t_0]_\T,\overline{B}(q_0,R))}{q}
{F(q)}
\end{equation*}
with
$$
{\fonction{F(q)}{[c,t_0]_\T}{\overline{B}(q_0,R)}{t}{q_0 - \di \int_{[t,t_0[_\T} f(q(\tau),\tau) \; \DD \tau.}}
$$
Since $F$ has a unique fixed point from the Banach fixed point theorem, we conclude that $q_1 = q_2$ on $[c,t_0]_\T$. Hence $\overline{t} < t_0$.

\item \textit{Fact 2: $q_1(\overline{t})=q_2(\overline{t})$.} If $\overline{t}$ is a right-scattered point of $\T$, then $\sigma(\overline{t})$ is a left-scattered point of $\T$ and $q_1(\sigma(\overline{t})) = q_2(\sigma(\overline{t}))$. As a consequence, $q_1(\overline{t}) = q_2(\overline{t}) = G^+(\overline{t})^{-1} (q_1(\sigma(\overline{t})))$. If $\overline{t}$ is a right-dense point of $\T$, then $q_1(\overline{t}) = q_2(\overline{t})$ from the continuity of $q_1$ and $q_2$ and since $q_1=q_2$ on $]\overline{t},t_0]_\T$.

\item \textit{Fact 3: $\overline{t} > a$.} Indeed, if $\overline{t} = a$ then $A = \emptyset$ since $q_1(\overline{t})=q_2(\overline{t})$;

\item \textit{Fact 4: $\overline{t}$ is a left-dense point of $\T$.} Indeed, if $\overline{t}$ were to be a left-scattered point of $\T$, since $q_1(\overline{t})=q_2(\overline{t})$, then $q_1(\rho(\overline{t}))=q_2(\rho(\overline{t})) = G^+(\rho(\overline{t}))^{-1}(q_1(\overline{t}))$ and then it would raise a contradiction with the definition of $\overline{t}$.
\end{enumerate}
Let us denote by $ \overline{x} = q_1(\overline{t})=q_2(\overline{t})$. Let $R$, $\delta$ and $L$ associated with $\overline{t}$ and $\overline{x}$ in \eqref{eqcondldloc-Lip}. Let $M$ associated with $B(\overline{x},R) \in \KK$ and $[\overline{t}-\delta,\overline{t}[_\T$ in \eqref{eqcondinfini}. Consider $0 <\delta_1 \leq \delta$ such that $c_0=\overline{t}-\delta_1 \in [a,\overline{t}[_\T$ and $\delta_1$ is sufficiently small in order to have $\delta_1 M \leq R$, $\delta_1 L < 1$ and $q_1$, $q_2 \in \CC ([c_0,\overline{t}]_\T,B(\overline{x},R))$. Since $q_1$ and $q_2$ are solutions of $\mathrm{(\DD\text{-}CP)}$ on $[a,b]_\T$, they are in particular fixed points of the $\delta_1 L$-contraction map
\begin{equation*}
\fonction{F_0}{\CC ([c_0,\overline{t}]_\T,B(\overline{x},R))}{\CC ([c_0,\overline{t}]_\T,B(\overline{x},R))}{q}{F_0(q)}
\end{equation*}
with
$$
{\fonction{F_0(q)}{[c_0,\overline{t}]_\T}{B(\overline{x},R)}{t}{\overline{x} - \di \int_{[t,\overline{t}[_\T} f(q(\tau),\tau) \; \DD \tau.}}
$$
Since $F_0$ has a unique fixed point from the Banach fixed point theorem, we conclude that $q_1 = q_2$ on $[c_0,\overline{t}]_\T$, and this is a contradiction. Consequently $A = \emptyset$. 

In the same way, we prove that $B = \emptyset$ and the proof is complete.
\end{proof}

Theorem~\ref{thm21-1} follows from Lemma~\ref{prop21-2}, Propositions~\ref{prop21-3} and \ref{prop21-4}.

\subsection{Proof of Theorem~\ref{thm21-2}}\label{app12}

\begin{proposition}\label{prop21-5}
Under the assumptions of Theorem~\ref{thm21-1}, let $(q,I_\T)$ be the maximal solution of $\mathrm{(\DD\text{-}CP)}$. Then either $I_\T = \T$, that is, the solution $(q,I_\T)$ is global, or
\begin{enumerate}
\item if $t_0 = \min \T$ then $I_\T = [t_0,b[_\T$ where $b \in ]t_0,+\infty[_\T$ is a left-dense point of $\T$;
\item if $t_0 = \max \T$ then $I_\T = ]a,t_0]_\T$ where $a \in ]-\infty,t_0[_\T$ is a right-dense point of $\T$;
\item if $t_0 \neq \inf \T$ and $t_0 \neq \sup \T$ then $I_\T = ]a,+\infty[_\T$ or $I_\T = ]-\infty,b[_\T$ or $I_\T = ]a,b[_\T$, where $a \in ]-\infty,t_0[_\T$ is a right-dense point of $\T$ and $b \in ]t_0,+\infty[_\T$ is a left-dense point of $\T$.
\end{enumerate}
\end{proposition}

\begin{proof}
We only prove this proposition in the first case of Theorem~\ref{thm21-1} (the other ones are derived similarly).

Let us first prove that if $I_\T = [t_0,b]_\T$ then $b=\max \T$ (and thus $I_\T = \T$). By contradiction, assume that $I_\T = [t_0,b]_\T$ with $b < \sup \T$. Consider the $\DD$-Cauchy problem 
$$z^\DD(t)  = f (z(t),t), \quad z(b)=q(b).$$
As in Proposition~\ref{prop21-3}, we can prove that it has a solution $(z,[b,b_1]_\T)$ with $b_1 \in ]b,+\infty[_\T$. Then, we define $q_1$ by
\begin{equation}
q_1 (t) = \left\lbrace \begin{array}{rcl}
q(t) & \text{if} & t \in [t_0,b]_\T, \\
z(t) & \text{if} & t \in [b,b_1]_\T,
\end{array} \right.
\end{equation}
for every $ t \in [t_0,b_1]_\T$.
Then $q_1 \in \CC([t_0,b_1]_\T)$ and one can easily prove that
\begin{equation*}
q_1 (t) = q_0 + \di \int_{[t_0,t[_\T} f(q_1(\tau),\tau) \; \DD \tau.
\end{equation*}
for every $ t \in [t_0,b_1]_\T$.
It follows from Lemma~\ref{prop21-1} that $(q_1,[t_0,b_1]_\T)$ is a solution of $\mathrm{(\DD\text{-}CP)}$ and is a strict extension of $(q,[t_0,b]_\T)$. It is a contradiction with the maximality of $(q,[t_0,b]_\T)$. 

If $I_\T = [t_0,b[_\T$ with $b$ a left-scattered point of $\T$, then $I_\T = [a,\rho(b)]_\T$ with $\rho(b) < \sup \T$ and we recover to the previous contradiction.
\end{proof}

\begin{lemma}\label{prop21-6}
Under the assumptions of Theorem~\ref{thm21-1}, let $(q,I_\T)$ be the maximal solution of $\mathrm{(\DD\text{-}CP)}$. If $(q,I_\T)$ is not global, then $q$ cannot be continuously extended with a value in $\Omega$ at $t=a$ or at $t=b$ (see Proposition~\ref{prop21-5} for $a$ and $b$).
\end{lemma}

\begin{proof}
We only prove this lemma in the first case of Theorem~\ref{thm21-1}.
By contradiction, let us assume that $q$ can be continuously extended with a value in $\Omega$ at $t=b$, that is, $\lim_{t \to b, \; t \in [t_0,b[_\T } q(t) = q_b \in \Omega$. Then, we define $q_1$ by
\begin{equation*}
q_1 (t) = \left\lbrace \begin{array}{rcl}
q(t) & \text{if} & t \in [t_0,b[_\T \\
q_b & \text{if}& t=b,
\end{array} \right.
\end{equation*}
for every $ t \in [t_0,b]_\T$.
In particular $q_1:[t_0,b]_\T\rightarrow\Omega$ and $q_1 \in \CC ([t_0,b]_\T)$. Our aim is to prove that $(q_1,[t_0,b]_\T)$ is a solution of $\mathrm{(\DD\text{-}CP)}$. 

Since $(q,[t_0,b[_\T)$ is a solution of $\mathrm{(\DD\text{-}CP)}$, it follows from Lemma~\ref{prop21-1} that
\begin{equation}\label{eq671}
q_1(t) = q(t) = q_0 + \di \int_{[t_0,t[_\T} f(q(\tau),\tau) \; \DD \tau = q_0 + \di \int_{[t_0,t[_\T} f(q_1(\tau),\tau) \; \DD \tau ,
\end{equation}
for every $b' \in ]t_0,b[_\T$ and every $t \in [t_0,b']_\T$.
Since $f(q_1,t) \in \L^1_\T ([t_0,b[_\T,\R^n)$ (see \eqref{eq666}), we infer from Lebesgue's dominated convergence theorem that
\begin{equation*}
q_1 (b) = q_b = q_0 + \di \int_{[t_0,b[_\T} f(q_1(\tau),\tau) \; \DD \tau.
\end{equation*}
Therefore \eqref{eq671} also holds for $b'=b$. It follows from Lemma~\ref{prop21-1} that $(q_1,[t_0,b]_\T)$ is a solution of $\mathrm{(\DD\text{-}CP)}$ and is a strict extension of $(q,[t_0,b[_\T)$. It is a contradiction with the maximality of $(q,[t_0,b[_\T)$.
\end{proof}

\begin{lemma}\label{prop21-7}
Under the assumptions of Theorem~\ref{thm21-1}, let $(q,I_\T)$ be the maximal solution of $\mathrm{(\DD\text{-}CP)}$. If $(q,I_\T)$ is not global, then for every $K \in \KK$ there exists $t\in I_\T$ (close to $a$ or $b$ depending on the cases listed in the theorem) such that $q(t)\in\Omega\setminus K$.
\end{lemma}

\begin{proof}
We only prove this lemma in the first case of Theorem~\ref{thm21-1}.
By contradiction, assume that there exists $K \in \KK $ such that $q$ takes its values in $K$ on $I_\T = [t_0,b[_\T$ with $b$ a left dense point of $\T$. Consider $M \geq 0$ associated with $K \in \KK$ and $[t_0,b[_\T$ in \eqref{eqcondinfini}. For all $t_1 \leq t_2$ elements of $[t_0,b[_\T$, one has
\begin{equation*}
\Vert q(t_2) - q(t_1) \Vert \leq \di \int_{[t_1,t_2[_\T} \Vert f(q(\tau),\tau) \Vert \; \DD \tau \leq M (t_2 - t_1).
\end{equation*}
Therefore $q$ is Lipschitz continuous and thus uniformly continuous on $[t_0,b[_\T$ with $b$ a left-dense point of $\T$. Hence $q$ can be continuously extended at $t=b$ with a value $q_b \in \R^n$. Moreover, since $q$ takes its values in the compact $K \subset \Omega$, it follows that $q_b \in \Omega$. Using Lemma~\ref{prop21-6}, this raises a contradiction. 
\end{proof}

The proof of Theorem~\ref{thm21-2} follows from Proposition~\ref{prop21-5} and Lemma~\ref{prop21-7}.

\subsection{Proof of Theorem~\ref{thm22-1}}\label{app13}
Note that since $\Omega = \R^n$ and since $f$ satisfies \eqref{eqcondgloblip}, $f$ automatically satisfies \eqref{eqcondfsta} and \eqref{eqcondrdloc-Lip}. Since $t_0 = \min \T$, $\mathrm{(\DD\text{-}CP)}$ admits a unique maximal solution $(q,I_\T)$ from Theorem~\ref{thm21-1}. Proving that $I_\T = \T$ requires the following result.

\begin{lemma}\label{lem22-1}
If $t_0 = \min \T$ then
\begin{equation*}
\int_{[t_0,t[_\T} (\tau - t_0)^k \; \DD \tau \leq \dfrac{(t-t_0)^{k+1}}{k+1},
\end{equation*}
for every $k \in \N$ and every $t \in \T$.
\end{lemma}

\begin{proof}
One has
\begin{equation*}
 \int_{[t_0,t[_\T} (\tau - t_0)^k \; \DD \tau =  \int_{[t_0,t[_\T} (\tau - t_0)^k \; d\tau + \di \sum_{r \in [t_0,t[_\T \cap \RR} \mu (r) (r -t_0)^k ,
\end{equation*}
for every $k \in \N$ and every $t \in \T$.
Since
\begin{equation*}
\begin{split}
\sum_{r \in [t_0,t[_\T \cap \RR} \mu (r) (r -t_0)^k &= \di \sum_{r \in [t_0,t[_\T \cap \RR} \di \int_{]r,\sigma(r)[} (r -t_0)^k \; d\tau \\
& \leq \di \sum_{r \in [t_0,t[_\T \cap \RR} \di \int_{]r,\sigma(r)[} (\tau -t_0)^k \; d\tau,
\end{split}
\end{equation*}
it follows that
\begin{equation*}
\di \int_{[t_0,t[_\T} (\tau - t_0)^k \; \DD \tau \leq \di \int_{[t_0,t[} (\tau - t_0)^k \; d\tau = \dfrac{(t-t_0)^{k+1}}{k+1},
\end{equation*}
and the proof is complete.
\end{proof}

We define the mapping
\begin{equation*}
\fonction{F}{\CC (\T,\R^n)}{\CC (\T,\R^n)}{q}{F(q)}
\end{equation*}
with
$$
{\fonction{F(q)}{\T}{\R^n}{t}{q_0 + \di \int_{[t_0,t[_\T} f ( q(\tau),\tau) \; \DD \tau.}}
$$
From Lemma~\ref{lem22-1}, one can easily prove by induction that
\begin{equation*}
\Vert F^k (q_1)(t) - F^k (q_2)(t) \Vert \leq \dfrac{L^k}{k!} \Vert q_1 - q_2 \Vert_\infty (t-t_0)^k,
\end{equation*}
for every $k\in \N^*$, all $q_1, q_2 \in \CC(\T,\R^n)$, and every $t \in \T$.
Then,
\begin{equation*}
\Vert F^k (q_1) - F^k (q_2) \Vert_\infty \leq \dfrac{(L(b-a))^k}{k!} \Vert q_1 - q_2 \Vert_\infty ,
\end{equation*}
for every $k\in \N^*$, all $q_1, q_2 \in \CC(\T,\R^n)$.
Therefore $F$ admits a contraction iterate and thus has a unique fixed point that is a global solution of $\mathrm{(\DD\text{-}CP)}$. This concludes the proof of Theorem~\ref{thm22-1}.

\subsection{Further comments for the shifted case}\label{app2}
An important remark in the \textit{shifted} case is the following. Let $(a,b) \in \T^2$ satisfying $\ab$ and let $q:[a,b]_\T\rightarrow\Omega$. Since $\sigma(t) \in [a,b]_\T$ for every $t \in [a,b[_\T$, $q^\sigma$ is well defined on $[a,b[_\T$. This remark permits to derive all results of Section~\ref{section2} in a similar way since $\DD$-integrals are considered on intervals of the form $[a,b[_\T$.

For example, if $f$ satisfies \eqref{eqcondinfini}, then for all $(a,b)\in\T^2$ such that $a<b$,
\begin{equation*}
f(q^\sigma,t) \in \L^\infty_\T ([a,b[_\T,\R^n) \subset \L^1_\T ([a,b[_\T,\R^n),
\end{equation*}
for every  $q \in \CC ([a,b]_\T,\R^n)$.
This remark permits to prove (from Section~\ref{section13}) the following $\DD$-integral characterization of the solutions of $\mathrm{(\DD\text{-}CP^\sigma)}$.

\begin{lemma}\label{prop31-1}
Let $I_\T \in \I$ and $q:I_\T\rightarrow\Omega$. If $f$ satisfies \eqref{eqcondinfini}, then the couple $(q,I_\T)$ is a solution of $\mathrm{(\DD\text{-}CP^\sigma)}$ if and only if for all $a,b \in I_\T$ satisfying $\ab$, one has $q \in \CC([a,b]_\T,\R^n)$ and
\begin{equation*}
q(t) = \left\lbrace \begin{array}{lcc}
q_0 + \int_{[t_0,t[_\T} f(q^\sigma (\tau),\tau) \; \DD \tau & \text{if} & t \geq t_0, \\
q_0 - \int_{[t,t_0[_\T} f(q^\sigma (\tau),\tau) \; \DD \tau & \text{if} & t \leq t_0. \\
\end{array} \right.
\end{equation*}
for every $t \in [a,b]_\T$.
\end{lemma}

All results permitting to prove Theorems \ref{thm31-1} and \ref{thm31-2} can be derived as in Section \ref{app1}. Nevertheless, in order to derive Theorem \ref{thm32-1}, the following result is required.

\begin{lemma}\label{lem32-1}
If $t_0 = \max \T$ then
\begin{equation*}
\int_{[t,t_0[_\T} (t_0-\sigma(\tau))^k \; \DD \tau \leq \dfrac{(t_0-t)^{k+1}}{k+1},
\end{equation*}
for every $k \in \N$ and every $t \in \T$.
\end{lemma}

\begin{proof}
One has
$$ \int_{[t,t_0[_\T} (t_0-\sigma(\tau))^k \; \DD \tau = \di \int_{[t,t_0[_\T} (t_0-\tau)^k \; d\tau + \di \sum_{r \in [t,t_0[_\T \cap \RR} \mu (r) (t_0-\sigma(r))^k ,$$
for every $k \in \N$ and every $t \in \T$.
Since
\begin{equation*}
\begin{split}
 \sum_{r \in [t,t_0[_\T \cap \RR} \mu (r) (t_0-\sigma(r))^k &= \di \sum_{r \in [t,t_0[_\T \cap \RR} \di \int_{]r,\sigma(r)[} (t_0-\sigma(r))^k \; d\tau \\
 &\leq \di \sum_{r \in [t,t_0[_\T \cap \RR} \di \int_{]r,\sigma(r)[} (t_0-\tau)^k \; d\tau,
\end{split}
\end{equation*}
we infer that
\begin{equation*}
\int_{[t,t_0[_\T} (t_0-\sigma(\tau))^k \; \DD \tau \leq \di \int_{[t,t_0[} (t_0-\tau)^k \; d\tau = \dfrac{(t_0-t)^{k+1}}{k+1},
\end{equation*}
and the statement follows.
\end{proof}

\bibliographystyle{plain}

\begin{thebibliography}{}

\end{thebibliography}


\begin{thebibliography}{10}

\bibitem{agar2}
R.P. Agarwal and M.~Bohner.
\newblock Basic calculus on time scales and some of its applications.
\newblock {\em Results Math.}, 35(1-2):3--22, 1999.

\bibitem{agar3}
R.P. Agarwal, M.~Bohner, and A.~Peterson.
\newblock Inequalities on time scales: a survey.
\newblock {\em Math. Inequal. Appl.}, 4(4):535--557, 2001.

\bibitem{agar}
R.P. Agarwal, V.~Otero-Espinar, K.~Perera, and D.R. Vivero.
\newblock Basic properties of {S}obolev's spaces on time scales.
\newblock {\em Adv. Difference Equ.}, pages Art. ID 38121, 14, 2006.

\bibitem{ati}
F.M. Atici, D.C. Biles, and A.~Lebedinsky.
\newblock An application of time scales to economics.
\newblock {\em Math. Comput. Modelling}, 43(7-8):718--726, 2006.

\bibitem{torr8}
Z.~Bartosiewicz and D.F.M. Torres.
\newblock Noether's theorem on time scales.
\newblock {\em J. Math. Anal. Appl.}, 342(2):1220--1226, 2008.

\bibitem{bohn}
M.~Bohner and A.~Peterson.
\newblock {\em Dynamic equations on time scales}.
\newblock Birkh\"auser Boston Inc., Boston, MA, 2001.
\newblock An introduction with applications.

\bibitem{bohn3}
M.~Bohner and A.~Peterson.
\newblock {\em Advances in dynamic equations on time scales}.
\newblock Birkh\"auser Boston Inc., Boston, MA, 2003.

\bibitem{caba2}
A.~Cabada and D.R. Vivero.
\newblock Criterions for absolute continuity on time scales.
\newblock {\em J. Difference Equ. Appl.}, 11(11):1013--1028, 2005.

\bibitem{caba}
A.~Cabada and D.R. Vivero.
\newblock Expression of the {L}ebesgue {$\Delta$}-integral on time scales as a
  usual {L}ebesgue integral: application to the calculus of
  {$\Delta$}-antiderivatives.
\newblock {\em Math. Comput. Modelling}, 43(1-2):194--207, 2006.

\bibitem{cich2}
M.~Cicho{\'n}, I.~Kubiaczyk, A.~Sikorska-Nowak, and A.~Yantir.
\newblock Weak solutions for the dynamic {C}auchy problem in {B}anach spaces.
\newblock {\em Nonlinear Anal.}, 71(7-8):2936--2943, 2009.

\bibitem{codd}
E.A. Coddington and N.~Levinson.
\newblock {\em Theory of ordinary differential equations}.
\newblock McGraw-Hill Book Company, Inc., New York-Toronto-London, 1955.

\bibitem{gama}
J.G.P. Gamarra and R.V. Solv\'e.
\newblock Complex discrete dynamics from simple continuous population models.
\newblock {\em Bull. Math. Biol.}, 64:611--620, 2002.

\bibitem{guse}
G.S. Guseinov.
\newblock Integration on time scales.
\newblock {\em J. Math. Anal. Appl.}, 285(1):107--127, 2003.

\bibitem{hilg}
S.~Hilger.
\newblock {\em Ein {M}a$\ss$kettenkalk\"ul mit {A}nwendungen auf
  {Z}entrumsmannigfaltigkeiten}.
\newblock PhD thesis, Universit\"at {W}\"urzburg, 1988.

\bibitem{hilg2}
S.~Hilger.
\newblock Analysis on measure chains---a unified approach to continuous and
  discrete calculus.
\newblock {\em Results Math.}, 18(1-2):18--56, 1990.

\bibitem{hils3}
R.~Hilscher and V.~Zeidan.
\newblock Time scale embedding theorem and coercivity of quadratic functionals.
\newblock {\em Analysis (Munich)}, 28(1):1--28, 2008.

\bibitem{hils2}
R.~Hilscher and V.~Zeidan.
\newblock Weak maximum principle and accessory problem for control problems on
  time scales.
\newblock {\em Nonlinear Anal.}, 70(9):3209--3226, 2009.

\bibitem{hils4}
R.~Hilscher, V.~Zeidan, and W.~Kratz.
\newblock Differentiation of solutions of dynamic equations on time scales with
  respect to parameters.
\newblock {\em Adv. Dyn. Syst. Appl.}, 4(1):35--54, 2009.

\bibitem{smal}
M.W. Hirsch and S.~Smale.
\newblock {\em Differential equations, dynamical systems, and linear algebra}.
\newblock Academic Press [A subsidiary of Harcourt Brace Jovanovich,
  Publishers], New York-London, 1974.
\newblock Pure and Applied Mathematics, Vol. 60.

\bibitem{kac}
V.~Kac and P.~Cheung.
\newblock {\em Quantum calculus}.
\newblock Universitext. Springer-Verlag, New York, 2002.

\bibitem{karp}
B.~Karpuz.
\newblock Existence and uniqueness of solutions to systems of delay dynamic
  equations on time scales.
\newblock {\em Int. J. Math. Comput.}, 10(M11):48--58, 2011.

\bibitem{kaym}
B.~Kaymak{\c{c}}alan.
\newblock Existence and comparison results for dynamic systems on a time scale.
\newblock {\em J. Math. Anal. Appl.}, 172(1):243--255, 1993.

\bibitem{kubi}
I.~Kubiaczyk and A.~Sikorska-Nowak.
\newblock Existence of solutions of the dynamic {C}auchy problem on infinite
  time scale intervals.
\newblock {\em Discuss. Math. Differ. Incl. Control Optim.}, 29:113--126, 2009.

\bibitem{laks}
V.~Lakshmikantham, S.~Sivasundaram, and B.~Kaymakcalan.
\newblock {\em Dynamic systems on measure chains}, volume 370 of {\em
  Mathematics and its Applications}.
\newblock Kluwer Academic Publishers Group, Dordrecht, 1996.

\bibitem{torr7}
A.B. Malinowska and D.F.M. Torres.
\newblock The delta-nabla calculus of variations.
\newblock {\em Fasc. Math.}, (44):75--83, 2010.

\bibitem{may}
R.M. May.
\newblock Simple mathematical models with very complicated dynamics.
\newblock {\em Nature}, 261:459--467, 1976.

\bibitem{tisd}
C.C. Tisdell and A.H. Zaidi.
\newblock Successive approximations to solutions of dynamic equations on time
  scales.
\newblock {\em Comm. Appl. Nonlinear Anal.}, 16(1):61--87, 2009.

\bibitem{zhan3}
Z.~Zhan and W.~Wei.
\newblock On existence of optimal control governed by a class of the
  first-order linear dynamic systems on time scales.
\newblock {\em Appl. Math. Comput.}, 215(6):2070--2081, 2009.

\end{thebibliography}

\end{document}